\documentclass[12pt,twoside]{amsart}
\usepackage{amsmath}
\usepackage{amsthm}
\usepackage{amsfonts}
\usepackage{amssymb}
\usepackage{latexsym}
\usepackage{mathrsfs}
\usepackage{amsmath}
\usepackage{amsthm}
\usepackage{amsfonts}
\usepackage{amssymb}
\usepackage{latexsym}
\usepackage{geometry}
\usepackage{dsfont}
\usepackage[dvips]{graphicx}
\usepackage{color}
\usepackage[all]{xy}

\date{}
\allowdisplaybreaks[4] \footskip=15pt
\renewcommand{\uppercasenonmath}[1]{}

\usepackage{graphicx,amssymb}
\usepackage[all]{xy}
\usepackage{amsmath}

\allowdisplaybreaks
\usepackage{amsthm}
\usepackage{color}

\theoremstyle{plain}
\newtheorem{theorem}{Theorem}[section]

\newtheorem{lemma}[theorem]{Lemma}

\newtheorem{example}[theorem]{Example}
\newtheorem*{open question}{Open Question}

\newtheorem{question}{Question}

\theoremstyle{definition}

\theoremstyle{remark}
\newtheorem{remark}[theorem]{Remark}

\def\GV{{\rm GV}}

\def\GV{{\rm GV}}

\begin{document}
\begin{center}
{\large  \bf On generalized fraction and power series properties of $\mathcal{S}$-Noetherian rings}

\vspace{0.5cm}   Xiaolei Zhang$^{a}$

{\footnotesize
School of Mathematics and Statistics, Shandong University of Technology,
Zibo 255049, China\\

E-mail: zxlrghj@163.com\\}
\end{center}

\bigskip
\centerline { \bf  Abstract}
\bigskip
\leftskip10truemm \rightskip10truemm \noindent

In this note, we study the generalized fraction properties and power series properties of $\mathcal{S}$-Noetherian rings. Actually, we answer two questions proposed in [A. Dabbabi, A. Benhissi,  Generalization of the $S$-Noetherian concept, {\it  Arch. Math.} (Brno) \textbf{59}(4) (2023) 307-314.] 
\vbox to 0.3cm{}

\leftskip0truemm \rightskip0truemm
\bigskip

\section{Introduction}
Throughout this note, all rings are always   commutative rings with identity. Let $A$ be a ring. We always denote by $A[[x]]$ the formal power series ring with coefficients in $A$,  $S$  a multiplicative subset of $A$, and $\mathcal{S}$ a multiplicative system of ideals of  $A$. For a subset $U$ of  an $A$-module $M$, we denote by $\langle U\rangle$ or $(U)A$ the $A$-submodule of $M$ generated by $U$.

In 2002, Anderson and Dumitrescu \cite{ad02} introduced the so-called $S$-Noetherian rings. An ideal $I$ of a ring $A$ is said to be  \emph{$S$-finite} if there is a finitely generated subideal $K$ of $I$  such that $sI\subseteq K$ for some $s\in S$. And a ring $A$ is called an \emph{$S$-Noetherian ring} if every ideal of $A$ is $S$-finite. The $S$-Noetherian version of Cohen's Theorem, Eakin-Nagata Theorem and Hilbert Basis Theorem of both polynomial  and power series  forms are given in \cite{ad02}. Some more works on $S$-Noetherian rings can be found in \cite{hh15,l15,lO14,lO15,l07}.

Recently, Dabbabi and Benhissi \cite{DB23} generalized the notion of $S$-Noetherian rings in terms of multiplicative systems of ideals of a given ring. Let $A$ be a ring and $\mathcal{S}$ be a multiplicative system of ideals of $A$. An ideal $I$ of $A$ is said to be  \emph{$\mathcal{S}$-finite} if there is a finitely generated subideal $F$ of $I$  such that $HI\subseteq F$ for some $H\in \mathcal{S}$.  A ring $A$ is called an \emph{$\mathcal{S}$-Noetherian ring} if every ideal of $A$ is $\mathcal{S}$-finite. Certainly if $\mathcal{S}$ is composed of principal ideals generated by elements in $S$, then $\mathcal{S}$-Noetherian rings are exactly  $S$-Noetherian rings. Moreover, $\mathcal{S}$-Noetherian version of Cohen's Theorem, Eakin-Nagata Theorem and Hilbert Basis Theorem of  polynomial form are also investigated in \cite{DB23}.

It is known in \cite{ad02} that if $A$ is an $S$-Noetherian ring, then the fraction ring $A_S$ is a Noetherian ring, and the power series ring $A[[x]]$ is also an $S$-Noetherian ring under the anti-Archimedean condition. However, they are left as two open questions for $\mathcal{S}$-Noetherian rings by  Dabbabi and Benhissi (see 
\cite[\emph{Questions}]{DB23}). The main motivation of this note is to investigate these two questions. Actually, we show that suppose $A$ is an $\mathcal{S}$-Noetherian domain. The generalized fraction ring $A_\mathcal{S}$ need not be Noetherian (see Example \ref{ce}). We also obtain that if $A$ is $\mathcal{S}$-Noetherian, then $A[[x]]$ is  $\mathcal{S}$-Noetherian  under some mild assumption (see Theorem \ref{pow}).

\section{main results}

Let $A$ be an integral domain with its quotient field $K$ and $\mathcal{S}$ a multiplicative system of ideals of $A$. Denote by $$A_\mathcal{S}=\{x\in K\mid xH\subseteq A\  \mbox{for some}\ H\in \mathcal{S}\},$$
and call it the generalized fraction ring of $A$ with respect to $\mathcal{S}$. If  $\mathcal{S}=\{sA\mid s\in S\}$ for some  multiplicative subset $S$ of $A$, then $A_\mathcal{S}=A_S$ the localization of $A$ at $S$. It follows by \cite[Proposition 2(f)]{ad02} that if  $A$ is an $S$-Noetherian ring, then  $A_S$ is a Noetherian ring. For general case, the authors in \cite{DB23} proposed the following question:

\begin{question}\label{1} Let $A$ be an integral domain and $\mathcal{S}$ a  multiplicative system of ideals of $A$ such that $A$ is $\mathcal{S}$-Noetherian. Does it follow that  $A_\mathcal{S}$ is Noetherian?
\end{question}

To give a counter-example to this question, we recall some basic notions of $w$-operations on integral domains (see \cite{wm97,fk16} for more details).

Let $A$ be an integral domain with its quotient field $K$. Let $J$ be a finitely generated ideal of $A$ and set $J^{-1}:=\{x\in K\mid Jx\subseteq A\}$. If $J^{-1}=A$. Then $J$ is said to be a $\GV$-ideal of $A$, and denoted it by  $J\in \GV(A)$. Certainly, $\GV(A)$ is a multiplicative system of ideals of $A$. Let $M$ be a torsion-free $A$-module. Denote by $$M_w=\{x\in M\otimes_AK\mid Jx\subseteq M\  \mbox{for some}\ J\in \GV(A)\}.$$
Trivially, $M\subseteq M_w$ and $(M_w)_w=M_w$. Moreover, if $M=M_w$, then $M$ is called a $w$-module. Trivially, the basic ring $A$ itself is a $w$-module. So if we take $\mathcal{S}=\GV(A)$, then $A_\mathcal{S}=A_w=A$. Now, we are ready to give a counter-example to Question \ref{1}.

\begin{example}\label{ce} Let $A$ be a non-Noetherian domain such that every $w$-ideal of $A$ is finitely generated $($e.g. $a$ is a non-Noetherian unique factorization domain, see \cite[Theorem 7.9.5]{fk16}$).$ We claim that $A$ is a $\GV(A)$-Noetherian ring in the sense of \cite{DB23}. Indeed, let $I$ be an ideal of $A$. Then $I_w$ is a finitely generated ideal of $A$ by assumption. Assume $I_w=\langle a_1,\dots,a_n\rangle$. Then there exists $J_i\in\GV(A)$ such that $J_ia_i\in I$ for each $i=1,\dots,n$. Setting  $J=J_1\dots J_n$, we have $J\in\GV(A)$ and $JI_w\subseteq I.$ Note that $JI\subseteq JI_w$ and $JI_w$ is finitely generated. So $I$ is $\GV(A)$-finite in the sense of \cite{DB23}. Consequently, $A$ is a $\GV(A)$-Noetherian ring. However, the generalized fraction ring $A_\mathcal{S}=A_w=A$ is not a Noetherian domain.
\end{example}

Let $A$ be a ring and $S$ a  multiplicative subset of $A$ such that for each $s\in S$, $\bigcap\limits_{n=1}^\infty s^nA$ contains some element in $S$ (i.e., $S$ is anti-Archimedean). Then it was proved in \cite[Proposition 10]{ad02} that if $A$ is an $S$-Noetherian ring, then  $A[[x]]$ is also  $S$-Noetherian. For the general case, the authors in \cite{DB23} proposed the following question: 
\begin{question} Let $A$ be a ring and $\mathcal{S}$ a  multiplicative system of ideals of $A$ such that for each $I\in \mathcal{S}$, $\bigcap\limits_{n=1}^\infty I^n$ contains some ideal of $\mathcal{S}$. Suppose $A$ is $\mathcal{S}$-Noetherian. Do it follow that $A[[x]]$ is also $\mathcal{S}$-Noetherian?
\end{question}
We  give a positive answer to this question under a mild assumption.
\begin{lemma}\label{fangmi}
	Let $H=\langle a_1, a_2,\cdots,a_n\rangle$ be a finitely generated ideal of $A$. Suppose $m$ is  a positive integer and $I=\langle a^m_1, a^m_2,\cdots,a^m_n\rangle$. Then $H^{mn}\subseteq I$.
\end{lemma}
\begin{proof}
	Note that $H^{mn}$ is generated by $\{\prod\limits_{i=1}^na_i^{k_i}\mid  \sum\limits_{i=1}^nk_i=mn\}$. By the pigeonhole principle, there exits some $k_i$ such that $k_i\geq m$. So each $\prod\limits_{i=1}^na_i^{k_i}\in I$, and thus $H^{mn}\subseteq I$.
\end{proof}
\begin{theorem}\label{pow} Let $A$ be a ring and $\mathcal{S}$ be a  multiplicative system of finitely generated ideals of $A$ such that for each $I\in \mathcal{S}$, $\bigcap\limits_{n=1}^\infty I^n$ contains some ideal of $\mathcal{S}$. Suppose $A$ is an $\mathcal{S}$-Noetherian ring. Then $A[[x]]$ is also an $\mathcal{S}$-Noetherian ring.
\end{theorem}
\begin{proof} It follows by \cite[Corollary 2.12]{DB23} that we only need to show that every prime ideal $P$ of $A[[x]]$ that not containing any ideal in  $\mathcal{S}$ is $\mathcal{S}$-finite. Let $\pi: A[[x]]\rightarrow A$ be an $A$-homomorphism sending $x$ to $0$. Set $P'=\pi(P)$. Since $A$ is $\mathcal{S}$-Noetherian, $HP'\subseteq (g_1(0),\dots,g_k(0))A$ for some $H=\langle h_l\mid l=1,\dots,n\rangle\in \mathcal{S}$ and $g_1,\dots,g_k\in P$.  If $x\in P$, then $P=(P',x)A[[x]]$, and so $HP\subseteq(g_1,\dots,g_k,x)A[[x]]\subseteq P$ implying $P$ is $\mathcal{S}$-finite. Now, we assume that $x\not\in P$. Let $f\in P$. Then $f(0)\in P'$, and so for each $l$, we have $h_lf(0)=\sum_i d_{0i,l}g_i(0)$  for some $d_{0i,l}\in A$. And so $h_lf-d_{0i,l}g_i$ can be written as $xf_{1,l}$ with $f_{1,l}\in A[[x]]$ for each $l$.  Note that $xf_{1,l}\in P$, and so $f_{1,l}\in P$ for each $l$ since $x\not\in P$ and $P$ is a prime ideal. In the same way, for each $l$ we can write $h_lf_1=\sum_i d_{1i,l}g_i+xf_2$ with  $d_{1i,l}\in A$ and $f_2\in P$. Continuing these steps, for each $l$ and $j$ we have $h_lf_j=\sum_i d_{ji,l}g_i+xf_{j+1}$ with  $d_{ji,l}\in A$ and $f_{j+1}\in P$, where $f_0=f$. So for each $l$, we have $$f=\sum_i g_i(\sum_j(d_{ji,l}/h_l^{j+1})x^j).$$ 
Let  $I\subseteq \bigcap\limits_{j=1}^\infty H^j$  be an ideal in $\mathcal{S}$. It follows by Lemma \ref{fangmi} that $H^{nj}\subseteq \langle h_1^j,\dots, h_n^j\rangle$ for each $j$, and so $I\subseteq\bigcap\limits_{j=1}^\infty \langle h_1^j,\dots, h_n^j\rangle$. Let $a=\sum\limits_{l=1}^na_{(j+1)l}h_l^{j+1}\in I$ with $a_{(j+1)l}\in A$ for each $j$ and $l$. Then 
\begin{align*}
af	=& (\sum\limits_{l=1}^na_{(j+1)l}h_l^{j+1}) (\sum_i g_i(\sum_j(d_{ji,l}/h_l^{j+1})x^j)) \\
= &\sum\limits_{l=1}^n(\sum_i g_i(\sum_j(a_{(j+1)l}d_{ji,l})x^j))\\
\in& (g_1,\dots,g_k)A[[x]].
\end{align*}
So we have  $If\subseteq (g_1,\dots,g_k)A[[x]]$. Hence $IP\subseteq (g_1,\dots,g_k)A[[x]]\subseteq P$. Therefore, $A[[x]]$ is an $\mathcal{S}$-Noetherian ring.	
\end{proof}

\begin{remark}	It follows from \cite[Theorem 2.7]{DB23} that suppose $\mathcal{S}$ is a  multiplicative system of ideals of $A$ such that for each $I\in \mathcal{S}$, $\bigcap\limits_{n=1}^\infty I^n$ contains some ideal of $\mathcal{S}$. Then $A$ is an $\mathcal{S}$-Noetherian ring implies that the polynomial ring $A[x]$ is also an $\mathcal{S}$-Noetherian ring. However, we don't know that whether it is also true for power series rings in this general situation. Note that in the proof of Theorem \ref{pow}, we use Cohen's Theorem for $\mathcal{S}$-Noetherian rings which necessarily requires that  ``$\mathcal{S}$ is a  multiplicative system of finitely generated ideals'' (see \cite[Example 2.14]{DB23}).
\end{remark}

\end{document}